\documentclass[11pt,a4paper]{amsproc}

\usepackage[hmargin=3.2cm,vmargin=3.4cm]{geometry}
\usepackage{amsmath,amssymb,mathrsfs}
\usepackage{enumitem}
\usepackage{hyperref}

\numberwithin{equation}{section}

\newtheorem{theorem}{Theorem}[section]
\newtheorem{corollary}[theorem]{Corollary}
\theoremstyle{definition}

\title[Finite local approximations of isometric actions]{On finite local approximations of isometric actions \\ of residually finite groups}

\author{Vadim Alekseev}
\address{Vadim Alekseev, TU Dresden, 01062 Dresden, Germany}
\email{vadim.alekseev@tu-dresden.de}

\author{Andreas Thom}
\address{Andreas Thom, TU Dresden, 01062 Dresden, Germany}
\email{andreas.thom@tu-dresden.de}

\date{\today}
\begin{document}

\begin{abstract}
We show that any isometric action of a residually finite group admits approximate local finite models. As a consequence, if $G$ is residually finite, every isometric $G$-action embeds isometrically into a metric ultraproduct of finite isometric $G$-actions.
\end{abstract}

\maketitle

\section{Introduction}

Let $G$ be a group and $\pi \colon F \to G$ be a surjection from a free group. A sequence of homomorphisms $\varphi_n \colon F \to {\rm Sym}(k_n)$ is a sofic approximation of $G$ if 
\[
\lim_{n \to \infty} \ell_{k_n}(\varphi_n(w))  =
\begin{cases}
0 &\text{if } \pi(w)=1_{G},\\
1 &\text{otherwise,}
\end{cases}
\]
where $\ell_{k_n}(\sigma):=k_n^{-1} |\{i \mid \sigma(i) \neq i\}|$ for $\sigma \in {\rm Sym}(k_n)$ is the normalized Hamming length, see \cite{PestovBriefGuide} for details. A sofic approximation gives rise to a sequence of isometric actions $F \times F \curvearrowright ({\rm Sym}(k_n),d_{k_n})$ by left-right multiplication, where $1_{k_n} \in {\rm Sym}(k_n)$ satisfies
\[
\lim_{n \to \infty} d_{k_n}(\psi_{(g,h)}(1_{k_n}),1_{k_n})  =
\begin{cases}
0 &\text{if } \pi(g)= \pi(h),\\
1 &\text{otherwise,}
\end{cases}
\]
where $d_{k_n}(\sigma,\tau):=\ell_{k_n}(\sigma^{-1}\tau)$ is the normalized Hamming metric. Thus, we may consider this sequence as a local finite approximation of the canonical isometric left-right action $F \times F \curvearrowright (G,\delta)$ with the $0$-$1$-valued metric $\delta.$ One motivation of this article was to clarify whether the mere existence of such a local approximation by finite isometric actions was already close to being a characterization of soficity of $G.$ To us, this seemed not impossible, since it is for example well-known that the associated unitary representation of $F \times F$ on $\ell^2 G$ is weakly contained in finite-dimensional unitary representations of $F \times F$ if and only if the group $G$ is hyperlinear, see \cite[Theorem 6.2.7, Exercise 6.2.4]{BrownOzawa}.
However, contrary to our expectation, we prove that this is not the case and every isometric action of a residually finite group admits approximate local finite models. 
In particular, such models of the left–right action $F \times F \curvearrowright (G,\delta)$ exist for all $\pi \colon F \to G$, independently of soficity of $G$. It remains an interesting challenge to put additional conditions on the approximating finite metric spaces that do imply soficity of the group $G$. A characterization of sofic groups in close analogy to Kirchberg's work and along similar lines is currently work in progress \cite{alekthom}.

\medskip

Our main result is the following theorem:
\begin{theorem} \label{thm:main}
Let $G$ be a residually finite group and let $G \curvearrowright_{\varphi} (X,d)$ be an isometric action. For all $X_0 \subset X$ finite, $A \subset G$ finite and $\varepsilon>0$, there exists a finite metric space $(Y,\eta)$ with an isometric action $G \curvearrowright_\psi (Y,\eta)$ and a map $f \colon X_0 \to Y$, such that
\[
\bigl|d(\varphi_g(x),\varphi_h(y)) - \eta(\psi_g(f(x)),\psi_h(f(y)))\bigr|\leq \varepsilon,
\quad \forall g,h \in A, \ x,y\in X_0.
\]
\end{theorem}

As a consequence of Theorem \ref{thm:main}, if $G$ is residually finite, every metric space with an isometric $G$-action embeds isometrically into a metric ultraproduct of finite isometric $G$-actions.

Our notion of local finite approximation is reminiscent of Rosendal's theory of
finitely approximable groups and actions. In \cite{RosendalRZ} he introduced the
Ribes--Zalesskii property (RZ) for a discrete group $G$ and showed that a group has
this property if and only if every isometric action of it on a metric space admits
\emph{finite approximations without $\varepsilon$}, in the sense that finite
configurations can be realized exactly inside finite $G$-spaces.

By contrast, Theorem~\ref{thm:main} shows that for residually finite groups one
always has \emph{approximate} local finite models for arbitrary isometric
actions, but we allow a small error~$\varepsilon$ on a prescribed finite
configuration. In particular, residual finiteness is sufficient for the
existence of such $\varepsilon$-models, whereas Rosendal's exact finite approximability characterizes the strictly smaller class of groups with
property~(RZ).

\section{Proof of the main result}

The following theorem is the main result. Recall, a group $G$ is \emph{residually finite} if for every $g\in G\setminus\{e\}$ there exists a finite group $Q$ and a homomorphism $\psi\colon G\to Q$ with $\psi(g)\neq e$. The proof is inspired
by Doucha's work on metric approximation of topological groups,
see in particular~\cite[Proposition 3.9]{DouchaMetricTopGroups}.

\begin{proof}[Proof of Theorem \ref{thm:main}]
Without loss of generality we may assume that $A=A^{-1}$ and $e \in A$.
Let $X_0\subset X$ be finite and set $Z:=G \times X_0.$
Define a pseudometric $\kappa$ on $Z$ by
$\kappa((g,x),(h,y)) := d(\varphi_g(x),\varphi_h(y)).$
Because the action is isometric, $\kappa$ is left $G$-invariant in the first coordinate, i.e.,
$\kappa((kg,x),(kh,y)) = \kappa((g,x),(h,y))$ for all $k,g,h\in G,\ x,y\in X_0$.
Let $d_{\rm st}\colon Z\times Z\to\{0,1\}$ be the discrete metric
\[
d_{\rm st}((g,x),(h,y)) =
\begin{cases}
0 &\text{if } (g,x)=(h,y),\\
1 &\text{otherwise.}
\end{cases}
\]
Define
\[
d_{\varepsilon}((g,x),(h,y)):=\kappa((g,x),(h,y)) + \varepsilon d_{\rm st}((g,x),(h,y)).
\]
Then $d_\varepsilon$ is a genuine metric on $Z$, still left-invariant in the first coordinate. In particular, for $(g,x)\neq(h,y)$ we have $d_\varepsilon((g,x),(h,y))\ge\varepsilon$.
Set $m:= \max \{d_{\varepsilon}((g,x),(h,y)) \mid x,y \in X_0,\ g,h \in A\}$
and choose $k\in \mathbb N$ with $k \geq 2\varepsilon^{-1}m.$
Now, using that $G$ is residually finite, fix a finite quotient
$ \pi\colon G\to H$
such that $\pi$ is injective on the finite set $B:=A^{k}$. Let $H_0:=\langle \pi(A)\rangle\le H$ be the subgroup generated by $\pi(A)$.

We set $Y_0:=H_0\times X_0$ and define a weighted graph structure on $Y_0$ as follows: we declare that there is an edge between the vertices
$(q\pi(g),x)$ and $(q\pi(h),y)$, for $q\in H_0$, $g,h\in A$, $x,y\in X_0$, with weight
$w\bigl((q\pi(g),x),(q\pi(h),y)\bigr)
:= d_\varepsilon((g,x),(h,y)).$

Let $\eta$ be the induced path metric on $Y_0$: for any two vertices, $\eta$ is the infimum of the sums of edge weights over all finite paths joining them. By construction, left multiplication by an element of $H_0$ on the first coordinate permutes vertices and preserves edge weights, so the action
$H_0\curvearrowright Y_0$, $h\cdot(q,x):=(hq,x)$,
is by isometries. We can extend this metric to $Y:=H \times X_0$ in a $H$-invariant way that sets all distances which are not fixed by $H$-invariance equal to a large constant. This way, 
we obtain an isometric action of  $G$ on $(Y,\eta)$ by
$\psi_g(q,x):=(\pi(g)q,x).$ Finally, define $f\colon X_0\to Y$ by $f(x):=(e,x)$ where $e\in H$ is the identity.

We need to compare $\eta$ and $d_\varepsilon$ on $A\times X_0$. We now show that for all $g,h\in A$, $x,y\in X_0$,
\begin{equation}\label{eq:eta-deps}
\eta\bigl(\psi_g(f(x)),\psi_h(f(y))\bigr)
= d_\varepsilon((g,x),(h,y)).
\end{equation}
This amounts to proving
\[
\eta\bigl((\pi(g),x),(\pi(h),y)\bigr) = d_\varepsilon((g,x),(h,y)).
\]

For $g,h\in A$ and $x,y\in X_0$ there is an edge between $(\pi(g),x)$ and $(\pi(h),y)$ corresponding to $q=e$, so
\[
\eta\bigl((\pi(g),x),(\pi(h),y)\bigr)
\le w\bigl((\pi(g),x),(\pi(h),y)\bigr)
= d_\varepsilon((g,x),(h,y)).
\]
We prove the reverse inequality by contradiction. Fix $g,h\in A$, $x,y\in X_0$ and suppose that
\begin{equation}\label{eq:contr}
\eta\bigl((\pi(g),x),(\pi(h),y)\bigr)
< d_\varepsilon((g,x),(h,y)).
\end{equation}
Let
\[
(\pi(g),x)=(v_0,x_0), (v_1,x_1),\dots,(v_n,x_n)=(\pi(h),y)
\]
be a path in $Y$ with minimal number $n$ of edges and such that
\[
\eta\bigl((\pi(g),x),(\pi(h),y)\bigr)
= \sum_{j=1}^n w_j,
\]
where $w_j$ is the weight of the $j$-th edge. Each edge comes from some triple $q_j\in H_0,g_j,h_j\in A$,
and has the form
\[
(v_{j-1},x_{j-1})=(q_j\pi(g_j),x_{j-1}),\qquad
(v_j,x_j)=(q_j\pi(h_j),x_j),
\]
with $w_j = d_\varepsilon((g_j,x_{j-1}),(h_j,x_j))$.
Since no edge is trivial, we have $w_j\ge\varepsilon$, hence
\[
\varepsilon n \le \sum_{j=1}^n w_j
= \eta\bigl((\pi(g),x),(\pi(h),y)\bigr)
< d_\varepsilon((g,x),(h,y))
\le m,
\]
and therefore $n < \varepsilon^{-1}m \le k/2$.
Next, consider the elements
\[
u := g^{-1}h,\qquad
u' := \prod_{j=1}^n g_j^{-1}h_j \in G.
\]
We claim that
\begin{equation}\label{eq:uuprime}
\pi(u) = \pi(u').
\end{equation}
Indeed, from $v_{j-1}=q_j\pi(g_j)$ and $v_j=q_j\pi(h_j)$ we get
$v_{j-1}^{-1}v_j = \pi(g_j^{-1}h_j),$
and hence
\[
v_n
= v_0 \prod_{j=1}^n v_{j-1}^{-1}v_j
= \pi(g)\,\prod_{j=1}^n \pi(g_j^{-1}h_j)
= \pi\Bigl(g \prod_{j=1}^n g_j^{-1}h_j\Bigr)
= \pi(gu').
\]
Since $v_n=\pi(h)$, this gives $\pi(gu')=\pi(h)$, i.e.~$\pi(u')=\pi(g^{-1}h)=\pi(u)$.

Now observe that $u = g^{-1}h \in A^2 \subset A^{k}=B$,
and each $g_j^{-1}h_j\in A^2$, so $u'=\prod_{j=1}^n g_j^{-1}h_j$ lies in $A^{2n}\subset A^{k}=B$, because $n\le k/2$. Thus $u,u'\in B$ and, by injectivity of $\pi$ on $B$, \eqref{eq:uuprime} implies $u'=u$, i.e.
\begin{equation}\label{eq:gprod= h}
g\prod_{j=1}^n g_j^{-1}h_j = h.
\end{equation}

Define elements $b_0,\dots,b_n\in G$ recursively by
$b_0:=g, b_j := b_{j-1}g_j^{-1}h_j\quad (1\le j\le n).$ Then by construction
\[
b_n = g\prod_{j=1}^n g_j^{-1}h_j = h
\]
by \eqref{eq:gprod= h}. We now compare $w_j$ with $d_\varepsilon$ along the sequence $(b_j,x_j)$. For each $j$ we have
$b_{j-1}^{-1}b_j = g_j^{-1}h_j.$ Using left-invariance of $d_\varepsilon$ we obtain
\begin{align*}
d_\varepsilon((b_{j-1},x_{j-1}),(b_j,x_j))
&= d_\varepsilon\bigl((e,x_{j-1}),(b_{j-1}^{-1}b_j,x_j)\bigr)\\
&= d_\varepsilon\bigl((e,x_{j-1}),(g_j^{-1}h_j,x_j)\bigr)\\
&= d_\varepsilon((g_j,x_{j-1}),(h_j,x_j))\\
&= w_j.
\end{align*}
Therefore
\[
\sum_{j=1}^n w_j
= \sum_{j=1}^n d_\varepsilon((b_{j-1},x_{j-1}),(b_j,x_j)).
\]
By the triangle inequality for $d_\varepsilon$ and the facts that $(b_0,x_0)=(g,x)$ and $(b_n,x_n)=(h,y)$, we get
\[
d_\varepsilon((g,x),(h,y))
\le \sum_{j=1}^n d_\varepsilon((b_{j-1},x_{j-1}),(b_j,x_j))
= \sum_{j=1}^n w_j
= \eta\bigl((\pi(g),x),(\pi(h),y)\bigr),
\]
which contradicts our assumption \eqref{eq:contr}. Hence no such path exists, and we must have
\[
\eta\bigl((\pi(g),x),(\pi(h),y)\bigr)
= d_\varepsilon((g,x),(h,y)).
\]
Finally, for $g,h\in A$ and $x,y\in X_0$ we have
\[
\eta(\psi_g(f(x)),\psi_h(f(y)))
= d_\varepsilon((g,x),(h,y))
= d(\varphi_g(x),\varphi_h(y)) + \varepsilon\, d_{\rm st}((g,x),(h,y)),
\]
and $d_{\rm st}\in\{0,1\}$, so
\[
\bigl|\eta(\psi_g(f(x)),\psi_h(f(y))) - d(\varphi_g(x),\varphi_h(y))\bigr|
\le \varepsilon.
\]
This completes the proof.
\end{proof}

A \emph{seminorm} on $G$ is a function $s\colon G\to[0,\infty)$ such that
$s(e)=0$, $s(g^{-1})=s(g)$, and $s(gh)\le s(g)+s(h)$ for all $g,h\in G$. A seminorm $s$ is called a \emph{norm} if $s(g)=0$ implies $g=e$.  A group together with a norm is called \emph{normed group}.
Define the \emph{standard} norm $n_{\mathrm{std}}\colon G\to\{0,1\}$ by
\[
n_{\mathrm{std}}(g) :=
\begin{cases}
0 & g=e,\\
1 & g\neq e.
\end{cases}
\]
There is a well-known correspondence between left-$G$-invariant metrics on $G$ and norms on $G$. Thus, the following corollary is immediate.

\begin{corollary}
\label{thm:coro}
Let $G$ be a residually finite group and let $s\colon G\to[0,\infty)$ be a seminorm.
Then, for every $A \subset G$ finite and $\varepsilon>0$, there exist a finite normed group $(H,\rho)$, and a surjective homomorphism $\varphi\colon G\to H$,
such that
\[
|\rho(\varphi(g)) - s(g)| \le \varepsilon
\qquad\text{for all }g\in A.
\]
\end{corollary}

Proposition 3.9 in \cite{DouchaMetricTopGroups} states the previous corollary in spirit for free groups, while the proof did not use freeness. Without going into the details, see for example \cite{BYBHU}, let us also record the following corollary:
\begin{corollary}
Let $G$ be a residually finite group. Every isometric action $G \curvearrowright (X,d)$ embeds isometrically into a metric ultraproduct of finite isometric $G$-actions.
\end{corollary}

\section*{Acknowledgments}

The second author thanks the Isaac Newton Institute for its hospitality during the workshop \emph{Stability and probabilistic methods} in November 2025. ChatGPT was used to assist in drafting parts of this manuscript. All content was reviewed and substantially revised by the authors, who are responsible for the final text. We thank Michal Doucha for interesting remarks on a first draft of this paper.

\end{document}